\documentclass[11pt]{article}
\usepackage[a4paper,margin=1in]{geometry}
\usepackage{amsmath,amssymb,amsthm,mathtools}
\usepackage{enumitem}
\usepackage[colorlinks=true,linkcolor=blue,citecolor=blue,urlcolor=blue]{hyperref}

\newtheorem{theorem}{Theorem}[section]
\newtheorem{lemma}[theorem]{Lemma}
\newtheorem{proposition}[theorem]{Proposition}

\theoremstyle{remark}
\newtheorem{remark}[theorem]{Remark}
\newtheorem{example}[theorem]{Example}

\newcommand{\qbinom}[2]{\genfrac{[}{]}{0pt}{}{#1}{#2}_{q}}
\newcommand{\A}{\mathcal A}

\title{Higher-Order Cyclotomic Congruences for $q$-Secant\\
and Generalized $q$-Euler Numbers}
\author{Jiang Zeng\thanks{College of Mathematics and Physics,
Wenzhou University, Wenzhou 325035, China; Universit\'e Claude Bernard
Lyon~1, CNRS UMR~5208, Institut Camille Jordan, France.
E-mail: \texttt{zeng@math.univ-lyon1.fr}.}}
\date{July 21, 2026}

\begin{document}
\maketitle

\begin{abstract}
Let $\A(2n)$ denote the set of up--down alternating permutations of
$\{1,2,\ldots,2n\}$, and let
\[
 E_{2n}(q)=\sum_{\sigma\in\A(2n)}q^{\operatorname{inv}(\sigma)}.
\]
Andrews and Foata proved that
$E_{2n}(q)\equiv q^{2n(n-1)}\pmod{(1+q)^2}$, and Liu recently obtained the
cubic refinement
\[
 E_{2n}(q)\equiv q^{2n(n-1)}-\binom n2(1+q)^2
 \pmod{(1+q)^3}.
\]
Using the reciprocal generating function for the $q$-secant numbers, a
third-order expansion of Gaussian coefficients at $q=-1$, finite
differences, and Newton interpolation, we prove the
fourth-order refinement
\[
 E_{2n}(q)\equiv q^{2n(n-1)}-\binom n2(1+q)^2
 +\binom n2(2n^2-2n-3)(1+q)^3
 \pmod{(1+q)^4}.
\]
More generally, the recurrence yields an effective procedure for computing
the expansion modulo $(1+q)^K$ for any prescribed $K$.  We then apply the
same local-expansion strategy to the generalized $q$-Euler numbers
$E_{pn\mid p}(q)$ of Sagan and Zhang.  For every prime $p$, we prove uniform
congruences modulo $[p]_q^3$ and $[p]_q^4$; the fourth-order term is governed
by a central $q$-Wolstenholme-type quotient associated with
${2p\brack p}_q$.  Thus the fourth-order secant congruence is the first case
of a general higher-cyclotomic method.
\end{abstract}



\section{Introduction}

A permutation $\sigma=\sigma_1\sigma_2\cdots\sigma_m$ is called
\emph{up--down alternating} if
\[
 \sigma_1<\sigma_2>\sigma_3<\sigma_4>\cdots.
\]
Write $\A(m)$ for the set of these permutations and define
\[
 E_{2n}(q)=\sum_{\sigma\in\A(2n)}q^{\operatorname{inv}(\sigma)}.
 \tag{1.1}\label{eq:defE}
\]
These polynomials are the inversion enumerators usually called the
$q$-secant numbers.  Their analytic definition is given by the reciprocal
generating function
\[
 \sum_{n\geq0}E_{2n}(q)\frac{u^{2n}}{(q;q)_{2n}}
 =
 \left(\sum_{n\geq0}(-1)^n\frac{u^{2n}}{(q;q)_{2n}}\right)^{-1}.
 \tag{1.2}\label{eq:genfun}
\]
Here $(a;q)_m=\prod_{j=0}^{m-1}(1-aq^j)$.

Andrews and Foata~\cite{AndrewsFoata} proved
\[
 E_{2n}(q)\equiv q^{2n(n-1)}\pmod{(1+q)^2}.
 \tag{1.3}\label{eq:AF}
\]
Liu~\cite{Liu} recently gave a combinatorial proof of the sharper congruence
\[
 E_{2n}(q)\equiv
 q^{2n(n-1)}-\binom n2(1+q)^2
 \pmod{(1+q)^3}.
 \tag{1.4}\label{eq:Liu}
\]
Our first purpose is to determine the next term in the $(1+q)$-adic
expansion and, more importantly, to place its computation in a systematic
higher-order framework.

\begin{theorem}\label{thm:main}
For every integer $n\geq0$,
\begin{align}
 E_{2n}(q)\equiv{}&q^{2n(n-1)}-\binom n2(1+q)^2 \notag\\
 &+\binom n2(2n^2-2n-3)(1+q)^3
 \pmod{(1+q)^4}.
 \tag{1.5}\label{eq:main}
\end{align}
\end{theorem}

The proof uses only \eqref{eq:genfun}.  Comparing coefficients in that
identity gives a convolution recurrence.  The other ingredient is the
Taylor expansion of $\qbinom{2n}{2k}$ at $q=-1$ through order three.
Our method also explains how the correction terms are discovered.  At each
order we first use the recurrence to prove that the unknown coefficient is
a polynomial in $n$ of bounded degree.  We then compute the required finite
set of initial values and recover the polynomial by Newton's interpolation
formula.

In Section~6 we extend the method to the generalized $q$-Euler
numbers of Sagan and Zhang.  The local parameter
$1+q=\Phi_2(q)$ is
replaced by $[p]_q=\Phi_p(q)$, and a block form of the $q$-Ljunggren
congruence replaces the expansion at $q=-1$.  For the fourth-order block
term we use Zudilin's expansion of Gaussian coefficients at roots of
unity.  We prove uniform congruences for $E_{pn\mid p}(q)$ modulo both
$[p]_q^3$ and $[p]_q^4$.
The latter result shows that the fourth-order correction is controlled by a
single central quotient attached to ${2p\brack p}_q$.  This uniform
cyclotomic statement, together with the arbitrary-order recurrence for the
$q$-secant numbers, is the principal structural contribution of the paper.

\section{A recurrence for the $q$-secant numbers}

Multiplying both sides of \eqref{eq:genfun} by the series on the right before
inversion and comparing coefficients of $u^{2n}$ yields the following
standard recurrence.

\begin{lemma}\label{lem:recurrence}
For $n\geq1$,
\[
 E_{2n}(q)=\sum_{k=1}^{n}(-1)^{k-1}
 \qbinom{2n}{2k}E_{2n-2k}(q),
 \qquad E_0(q)=1.
 \tag{2.1}\label{eq:recurrence}
\]
\end{lemma}

\begin{proof}
The coefficient of $u^{2n}$ in the product in question is
\[
 \sum_{k=0}^{n}
 \frac{(-1)^k}{(q;q)_{2k}}
 \frac{E_{2n-2k}(q)}{(q;q)_{2n-2k}}.
\]
It vanishes for $n\geq1$.  Multiplication by $(q;q)_{2n}$ and isolation of
the term $k=0$ give \eqref{eq:recurrence}.
\end{proof}

\section{Gaussian coefficients near \texorpdfstring{$q=-1$}{q=-1}}

Set
\[
 t=1+q,\qquad q=-1+t.
\]
We first record the expansion needed in the proof of Theorem~\ref{thm:main}.

\begin{lemma}\label{lem:qbinom-expansion}
Let $0\leq k\leq n$ and put $\ell=n-k$.  Then
\begin{align}
 \qbinom{2n}{2k}\bigg|_{q=-1+t}
 =\binom nk\bigl(1+a_{n,k}t+b_{n,k}t^2+c_{n,k}t^3\bigr)+O(t^4),
 \tag{3.1}\label{eq:qbinomexp}
\end{align}
where
\begin{align}
 a_{n,k}&=-2k\ell,                                           \tag{3.2}\label{eq:a}\\
 b_{n,k}&=\frac{k\ell(12k\ell+4n-5)}6,                     \tag{3.3}\label{eq:b}\\
 c_{n,k}&=\frac{k\ell(4n-3)}6
 -\frac{k^2\ell^2(4n-5)}3-\frac43k^3\ell^3.              \tag{3.4}\label{eq:c}
\end{align}
\end{lemma}

\begin{proof}
Separate the even and odd factors in the factorials:
\begin{align}
 \qbinom{2n}{2k}
 ={n\brack k}_{q^2}
 \frac{\prod_{j=1}^{n}(1-q^{2j-1})}
 {\prod_{j=1}^{k}(1-q^{2j-1})
  \prod_{j=1}^{\ell}(1-q^{2j-1})}.
 \tag{3.5}\label{eq:evenodd}
\end{align}
For $Q=1+h$, direct expansion gives
\begin{align}
 \log\frac{1-Q^m}{-mh}
 ={}&\frac{m-1}{2}h
 +\frac{(m-1)(m-5)}{24}h^2 \notag\\
 &+\frac{(m-1)(3-m)}{24}h^3+O(h^4).
 \tag{3.6}\label{eq:logeven}
\end{align}
Summing \eqref{eq:logeven} over the numerator and denominator factors of
${n\brack k}_Q$, and then using $Q=q^2=1-2t+t^2$, gives
\begin{align}
 \log\frac{{n\brack k}_{q^2}}{\binom nk}
 =-k\ell t+\frac{k\ell(n-2)}6t^2
 +\frac{k\ell(n-1)}6t^3+O(t^4).
 \tag{3.7}\label{eq:evenpart}
\end{align}
For odd $r$, one similarly obtains
\begin{align}
 \log\frac{1-(-1+t)^r}{2}
 =-\frac r2t+\frac{r(r-2)}8t^2
 +\frac{r(3r-4)}{24}t^3+O(t^4).
 \tag{3.8}\label{eq:logodd}
\end{align}
Put
\[
 F_m(t)=\sum_{j=1}^{m}
 \log\frac{1-(-1+t)^{2j-1}}{2}.
\]
The logarithm of the odd-factor quotient in \eqref{eq:evenodd} is
$F_n(t)-F_k(t)-F_\ell(t)$.  Substituting \eqref{eq:logodd}, using
$n=k+\ell$, and evaluating the resulting power sums gives
\begin{align*}
 [t](F_n-F_k-F_\ell)
 &=-\frac12(n^2-k^2-\ell^2)=-k\ell,\\
 [t^2](F_n-F_k-F_\ell)
 &=\frac{k\ell(n-1)}2,\\
 [t^3](F_n-F_k-F_\ell)
 &=\frac{k\ell(3n-2)}6.
\end{align*}
Here one uses
\[
 \sum_{j=1}^{m}(2j-1)=m^2,
\]
together with the standard formulas for $\sum_{j=1}^{m}j$ and
$\sum_{j=1}^{m}j^2$.  Therefore the odd-factor quotient has logarithm
\begin{align}
 -k\ell t+\frac{k\ell(n-1)}2t^2
 +\frac{k\ell(3n-2)}6t^3+O(t^4).
 \tag{3.9}\label{eq:oddpart}
\end{align}
Adding \eqref{eq:evenpart} and \eqref{eq:oddpart}, we find
\begin{align}
 \log\left(
 \frac{\qbinom{2n}{2k}|_{q=-1+t}}{\binom nk}
 \right)
 =-2k\ell t+\frac{k\ell(4n-5)}6t^2
 +\frac{k\ell(4n-3)}6t^3+O(t^4).
 \tag{3.10}\label{eq:logtotal}
\end{align}
Let
\[
 L(t)=At+Bt^2+Ct^3+O(t^4).
\]
Using
\[
 e^{L(t)}
 =1+L(t)+\frac{L(t)^2}{2}+\frac{L(t)^3}{6}+O(t^4),
\]
we have
\begin{equation}
 e^{L(t)}
 =1+At+\left(B+\frac{A^2}{2}\right)t^2
 +\left(C+AB+\frac{A^3}{6}\right)t^3+O(t^4).
 \tag{$*$}\label{eq:exponential-expansion}
\end{equation}

Exponentiating \eqref{eq:logtotal} and using
\eqref{eq:exponential-expansion} gives
\eqref{eq:qbinomexp}--\eqref{eq:c}.
\end{proof}

\section{Finite differences and discovery of the corrections}

We now describe a systematic way to determine the correction terms.  The
recurrence first implies that each correction coefficient, regarded as a
function of $n$, has bounded polynomial degree.  Once this degree bound is
known, the coefficient is determined by finitely many initial values, which
are most conveniently organized by finite differences.

Put
\[
 d_n=2n(n-1),\qquad Q_n(t)=(-1+t)^{d_n}=(1-t)^{d_n}.
 \tag{4.1}\label{eq:Qn}
\]
We shall repeatedly use the elementary binomial identity
\[
 \sum_{k=0}^{n}(-1)^k\binom nk k^j=0
 \qquad(0\leq j<n).
 \tag{4.2}\label{eq:bindiff}
\]
This identity expresses the fact that the $n$th finite difference of a
polynomial of degree less than $n$ vanishes.

More precisely, let $\Delta$ denote the forward-difference operator on a
sequence $(u_n)_{n\geq0}$:
\[
 \Delta u_n=u_{n+1}-u_n,
 \qquad
 \Delta^{r+1}u_n=\Delta(\Delta^r u_n).
\]
Repeated application of the definition gives
\[
 \Delta^r u_0=
 \sum_{j=0}^{r}(-1)^{r-j}\binom rj u_j.
 \tag{4.3}\label{eq:finite-degree}
\]
If $\Delta^r u_0=0$ for every $r>d$, Newton's interpolation formula gives
\[
 u_n=\sum_{r=0}^{d}\binom nr\Delta^r u_0.
 \tag{4.4}\label{eq:newton}
\]
Consequently, $u_n$ is a polynomial in $n$ of degree at most $d$.
Conversely, every polynomial sequence of degree at most $d$ has vanishing
finite differences of order greater than $d$.

This provides the guiding strategy below: we use the recurrence and
\eqref{eq:bindiff} to prove a degree bound for each unknown correction,
compute enough initial values, and then recover the correction for all $n$
from \eqref{eq:newton}.

\subsection{Recovering the quadratic correction}

Define $U_n$ by
\[
 E_{2n}(-1+t)=Q_n(t)+U_nt^2+O(t^3).
 \tag{4.5}\label{eq:Un-def}
\]
This definition is possible because the Andrews--Foata congruence
\eqref{eq:AF} says that the constant and linear coefficients agree with
those of $Q_n(t)$.

Let $\ell=n-k$.  Multiplying the expansion in
Lemma~\ref{lem:qbinom-expansion} by
\[
 Q_\ell(t)=1-2\ell(\ell-1)t
 +\ell(\ell-1)\bigl(2\ell(\ell-1)-1\bigr)t^2+O(t^3)
\]
gives
\begin{align}
 \qbinom{2n}{2k}\bigg|_{q=-1+t}Q_\ell(t)
 =\binom nk\bigl(1+A_{n,k}t+B_{n,k}t^2\bigr)+O(t^3),
 \tag{4.6}\label{eq:AB-expansion}
\end{align}
where
\[
 A_{n,k}=-2\ell(n-1)
\]
and
\begin{align}
 B_{n,k}={}&
 \frac{k\ell(12k\ell+4n-5)}6
 +4k\ell^2(\ell-1)\notag\\
 &+\ell(\ell-1)\bigl(2\ell(\ell-1)-1\bigr).
 \tag{4.7}\label{eq:Bnk}
\end{align}
For fixed $n$, $B_{n,k}$ is a polynomial in $k$ of degree at most four,
and
\[
 B_{n,0}=\binom{d_n}{2}=[t^2]Q_n(t).
\]

Substitution of \eqref{eq:Un-def} into \eqref{eq:recurrence}, followed by
comparison of the coefficients of $t^2$, yields
\begin{align}
 \binom{d_n}{2}+U_n
 =\sum_{k=1}^{n}(-1)^{k-1}\binom nk
 \bigl(B_{n,k}+U_{n-k}\bigr).
 \tag{4.8}\label{eq:Un-recurrence}
\end{align}
For $n\geq5$, identity \eqref{eq:bindiff} annihilates the polynomial
$B_{n,k}$.  Equation \eqref{eq:Un-recurrence} therefore becomes
\[
 \sum_{k=0}^{n}(-1)^k\binom nk U_{n-k}=0,
\]
or, equivalently,
\[
 \Delta^nU_0=0\qquad(n\geq5).
 \tag{4.9}\label{eq:Un-degree}
\]
Thus $U_n$ is a polynomial in $n$ of degree at most four.

The recurrence gives the five initial values
\[
 U_0,U_1,U_2,U_3,U_4=0,0,-1,-3,-6.
\]
Their difference table is
\[
\begin{array}{c|rrrrr}
n&0&1&2&3&4\\ \hline
U_n&0&0&-1&-3&-6\\
\Delta U_n&0&-1&-2&-3\\
\Delta^2U_n&-1&-1&-1\\
\Delta^3U_n&0&0\\
\Delta^4U_n&0
\end{array}
\]
Newton's formula \eqref{eq:newton} now gives
\[
 U_n=-\binom n2.
 \tag{4.10}\label{eq:Un}
\]
In particular, the analytic recurrence recovers Liu's cubic congruence and
also explains how the coefficient $\binom n2$ can be discovered.

\subsection{Recovering the cubic correction}

Having determined $U_n$, define $V_n$ by
\[
 E_{2n}(-1+t)=Q_n(t)-\binom n2t^2+V_nt^3+O(t^4).
 \tag{4.11}\label{eq:Vn-def}
\]
For fixed $n$, expand
\begin{align}
 &\frac{1}{\binom nk}
 \qbinom{2n}{2k}\bigg|_{q=-1+t}
 \left(Q_{n-k}(t)-\binom{n-k}{2}t^2\right)
 \notag\\
 &\hspace{20mm}
 =1+A_{n,k}t+B'_{n,k}t^2+D_{n,k}t^3+O(t^4).
 \tag{4.12}\label{eq:Dnk-expansion}
\end{align}
Lemma~\ref{lem:qbinom-expansion} shows that $D_{n,k}$ is a polynomial in
$k$ of degree at most six.  Its explicit expression is unnecessary; the
degree bound is the essential point.  Moreover,
\[
 D_{n,0}=[t^3]Q_n(t)=-\binom{d_n}{3}.
\]
Substitution of \eqref{eq:Vn-def} into \eqref{eq:recurrence} and comparison
of the coefficients of $t^3$ show, exactly as above, that
\[
 \Delta^nV_0=0\qquad(n\geq7).
 \tag{4.13}\label{eq:Vn-degree}
\]
Hence $V_n$ is a polynomial in $n$ of degree at most six.

The recurrence gives
\[
 V_0,V_1,V_2,V_3,V_4,V_5,V_6
 =0,0,1,27,126,370,855.
\]
The corresponding difference table is
\[
\begin{array}{c|rrrrrrr}
n&0&1&2&3&4&5&6\\ \hline
V_n&0&0&1&27&126&370&855\\
\Delta V_n&0&1&26&99&244&485\\
\Delta^2V_n&1&25&73&145&241\\
\Delta^3V_n&24&48&72&96\\
\Delta^4V_n&24&24&24\\
\Delta^5V_n&0&0\\
\Delta^6V_n&0
\end{array}
\]
Newton's formula therefore yields
\begin{align}
 V_n
 &=\binom n2+24\binom n3+24\binom n4\notag\\
 &=\binom n2(2n^2-2n-3).
 \tag{4.14}\label{eq:Vn}
\end{align}
Together, \eqref{eq:Vn-def} and \eqref{eq:Vn} prove
Theorem~\ref{thm:main}.

\begin{remark}
The same discovery principle works at every fixed order.  The coefficient
of $t^r$ in a normalized summand of \eqref{eq:recurrence} is a polynomial
in $k$ of bounded degree.  The binomial finite-difference identity then
proves a degree bound for the unknown coefficient as a polynomial in $n$.
After computing the necessary initial values, Newton interpolation recovers
the coefficient explicitly.
\end{remark}

\section{A general recurrence modulo \texorpdfstring{$(1+q)^K$}{(1+q)K}}

The preceding calculation can be organized as a recurrence that computes
$E_{2n}(q)$ modulo any prescribed power of $1+q$.  Once again put
\[
 t=1+q,\qquad q=-1+t,
\]
and fix an integer $K\geq1$.  Write
\[
 E_{2n}(-1+t)\equiv\sum_{r=0}^{K-1}e_{n,r}t^r\pmod{t^K}
 \tag{5.1}\label{eq:e-coefficients}
\]
and
\[
 {N\brack M}_{-1+t}
 \equiv\sum_{r=0}^{K-1}b_{N,M,r}t^r\pmod{t^K}.
 \tag{5.2}\label{eq:b-coefficients}
\]

\begin{proposition}\label{prop:general-recurrence}
The coefficients in \eqref{eq:b-coefficients} satisfy
\begin{align}
 b_{N,M,r}={}&b_{N-1,M,r}
 +\sum_{s=0}^{r}(-1)^{N-M-s}
 \binom{N-M}{s}b_{N-1,M-1,r-s},
 \tag{5.3}\label{eq:b-recurrence}
\end{align}
where $\binom ds=0$ for $s>d$.  The coefficients of the $q$-secant
polynomials then satisfy
\begin{align}
 e_{n,r}
 =\sum_{j=1}^{n}(-1)^{j-1}
 \sum_{s=0}^{r}b_{2n,2j,s}e_{n-j,r-s}.
 \tag{5.4}\label{eq:e-recurrence}
\end{align}
The initial conditions are
\[
 b_{N,0,0}=b_{N,N,0}=1,
 \qquad
 b_{N,0,r}=b_{N,N,r}=0\quad(r\geq1),
 \tag{5.5}\label{eq:b-initial}
\]
together with $b_{N,M,r}=0$ if $M<0$ or $M>N$, and
\[
 e_{0,0}=1,\qquad e_{0,r}=0\quad(r\geq1).
 \tag{5.6}\label{eq:e-initial}
\]
Consequently, \eqref{eq:b-recurrence} and \eqref{eq:e-recurrence} compute
$E_{2n}(q)$ modulo $(1+q)^K$ using integer arithmetic only.
\end{proposition}

\begin{proof}
The $q$-Pascal recurrence is
\[
 {N\brack M}_q
 ={N-1\brack M}_q+q^{N-M}{N-1\brack M-1}_q.
 \tag{5.7}\label{eq:q-pascal}
\]
Substituting $q=-1+t$ and using
\[
 (-1+t)^d=\sum_{s=0}^{d}(-1)^{d-s}\binom ds t^s
\]
in \eqref{eq:q-pascal}, comparison of the coefficients of $t^r$ gives
\eqref{eq:b-recurrence}.  Equation \eqref{eq:e-recurrence} follows by
substituting \eqref{eq:e-coefficients} and \eqref{eq:b-coefficients} into
the $q$-secant recurrence \eqref{eq:recurrence} and comparing coefficients
of $t^r$.
\end{proof}

There is also a compact polynomial form of the same algorithm.  Work in the
truncated ring
\[
 R_K=\mathbb Z[t]/(t^K)
\]
and put
\[
 B_{N,M}^{[K]}(t)={N\brack M}_{-1+t}\pmod{t^K}.
\]
Then
\begin{align}
 B_{N,M}^{[K]}(t)
 \equiv{}&B_{N-1,M}^{[K]}(t)
 +(-1+t)^{N-M}B_{N-1,M-1}^{[K]}(t)
 \pmod{t^K},
 \tag{5.8}\label{eq:B-polynomial-recurrence}
\end{align}
with
\[
 B_{N,0}^{[K]}(t)=B_{N,N}^{[K]}(t)=1.
\]
Starting with $S_0^{[K]}(t)=1$, define
\begin{align}
 S_n^{[K]}(t)
 \equiv\sum_{j=1}^{n}(-1)^{j-1}
 B_{2n,2j}^{[K]}(t)S_{n-j}^{[K]}(t)
 \pmod{t^K}.
 \tag{5.9}\label{eq:S-polynomial-recurrence}
\end{align}
It follows immediately that
\[
 E_{2n}(q)\equiv S_n^{[K]}(1+q)\pmod{(1+q)^K}.
 \tag{5.10}\label{eq:S-result}
\]
Thus \eqref{eq:B-polynomial-recurrence}--\eqref{eq:S-result} give a direct
algorithm at arbitrary order, even when no closed formula for the
coefficients is available.

\begin{example}[The case $K=5$]\label{ex:K5}
Apply \eqref{eq:B-polynomial-recurrence} and
\eqref{eq:S-polynomial-recurrence} in the ring
$R_5=\mathbb Z[t]/(t^5)$.  Writing the result relative to the leading
monomial $(-1+t)^{2n(n-1)}$, the recurrence gives successively
\[
 [t^2]\bigl(S_n^{[5]}(t)-(-1+t)^{2n(n-1)}\bigr)=-\binom n2,
\]
\[
 [t^3]\bigl(S_n^{[5]}(t)-(-1+t)^{2n(n-1)}\bigr)
 =\binom n2(2n^2-2n-3),
\]
and
\begin{align}
 [t^4]\bigl(S_n^{[5]}(t)-(-1+t)^{2n(n-1)}\bigr)
 =-\frac{n(n-1)(n-2)}{24}
 (24n^3-65n-65).
 \tag{5.11}\label{eq:K5-fourth-coefficient}
\end{align}
Here are more details on the computation of the last coefficient.  Write
\[
 S_n^{[5]}(t)=\sum_{r=0}^{4}e_{n,r}t^r
 \qquad\text{in }R_5
\]
and set
\[
 W_n=[t^4]\bigl(S_n^{[5]}(t)-(-1+t)^{d_n}\bigr),
 \qquad d_n=2n(n-1).
 \tag{5.12}\label{eq:Wn-def}
\]
Since $d_n$ is even,
\[
 [t^4](-1+t)^{d_n}=[t^4](1-t)^{d_n}=\binom{d_n}{4},
\]
and hence
\[
 W_n=e_{n,4}-\binom{d_n}{4}.
 \tag{5.13}\label{eq:Wn-from-e}
\]
The coefficients $e_{n,4}$ are obtained directly from
\eqref{eq:e-recurrence}:
\[
 e_{n,4}=\sum_{j=1}^{n}(-1)^{j-1}
 \sum_{s=0}^{4}b_{2n,2j,s}e_{n-j,4-s},
 \tag{5.14}\label{eq:en4-recurrence}
\]
where the $b_{N,M,s}$ are computed beforehand from
\eqref{eq:b-recurrence}.  This gives the following table.
\[
\begin{array}{c|r|r|r}
n&\binom{d_n}{4}&e_{n,4}&W_n\\ \hline
0&0&0&0\\
1&0&0&0\\
2&1&1&0\\
3&495&398&-97\\
4&10626&9415&-1211\\
5&91390&84865&-6525\\
6&487635&463990&-23645\\
7&1929501&1862021&-67480\\
8&6210820&6046978&-163842
\end{array}
\]

The same degree argument as in Section~4 shows that the normalized
coefficient of $t^4$ in a summand of \eqref{eq:recurrence} is a polynomial
in the summation index of degree at most eight.  Consequently,
\[
 \Delta^nW_0=0\qquad(n\geq9),
\]
so $W_n$ is a polynomial in $n$ of degree at most eight.  Its difference
table is
\[
\begin{array}{c|rrrrrrrrr}
n&0&1&2&3&4&5&6&7&8\\ \hline
W_n&0&0&0&-97&-1211&-6525&-23645&-67480&-163842\\
\Delta W_n&0&0&-97&-1114&-5314&-17120&-43835&-96362\\
\Delta^2W_n&0&-97&-1017&-4200&-11806&-26715&-52527\\
\Delta^3W_n&-97&-920&-3183&-7606&-14909&-25812\\
\Delta^4W_n&-823&-2263&-4423&-7303&-10903\\
\Delta^5W_n&-1440&-2160&-2880&-3600\\
\Delta^6W_n&-720&-720&-720\\
\Delta^7W_n&0&0\\
\Delta^8W_n&0
\end{array}
\]
The vanishing of the seventh and eighth differences shows that the actual
degree is at most six.  Applying Newton's formula with base point $n=1$
gives
\begin{align}
 W_n={}&-97\binom{n-1}{2}-920\binom{n-1}{3}
 -2263\binom{n-1}{4}\notag\\
 &\hspace{35mm}
 -2160\binom{n-1}{5}-720\binom{n-1}{6}.
 \tag{5.15}\label{eq:K5-newton}
\end{align}
Indeed, the coefficients in \eqref{eq:K5-newton} are respectively
$\Delta^2W_1,\ldots,\Delta^6W_1$; the constant and linear Newton
coefficients vanish because $W_1=W_2=0$.  Simplifying
\eqref{eq:K5-newton} yields
\[
 W_n=-\frac{n(n-1)(n-2)(24n^3-65n-65)}{24},
\]
which proves \eqref{eq:K5-fourth-coefficient}.
Consequently,
\begin{align}
 E_{2n}(q)\equiv{}&q^{2n(n-1)}-\binom n2(1+q)^2\notag\\
 &+\binom n2(2n^2-2n-3)(1+q)^3\notag\\
 &-\frac{n(n-1)(n-2)(24n^3-65n-65)}{24}(1+q)^4
 \pmod{(1+q)^5}.
 \tag{5.16}\label{eq:mod-fifth}
\end{align}
This illustrates that the general recurrence determines the next congruence
even when the required expansion of the Gaussian coefficient is not written
out explicitly.
\end{example}

\section{Generalized \texorpdfstring{$q$}{q}-secant numbers}
\label{sec:generalized}

Sagan and Zhang~\cite{SaganZhang} defined
\[
 E_{N\mid p}(q)=
 \sum_{\substack{\pi\in\mathfrak S_N\\
 \operatorname{Des}(\pi)=\{p,2p,3p,\ldots\}}}
 q^{\operatorname{inv}(\pi)}.
\]
The subsequence
\[
 \mathcal E_n^{(p)}(q)=E_{pn\mid p}(q)
\]
is the natural generalized $q$-secant sequence; for $p=2$ it is precisely
$E_{2n}(q)$.  Descent-set inclusion--exclusion gives
\begin{align}
 \sum_{n\geq0}\mathcal E_n^{(p)}(q)
 \frac{x^{pn}}{(q;q)_{pn}}
 =\left(\sum_{n\geq0}(-1)^n
 \frac{x^{pn}}{(q;q)_{pn}}\right)^{-1},
 \tag{6.1}\label{eq:generalized-genfun}
\end{align}
and hence
\begin{align}
 \mathcal E_n^{(p)}(q)
 =\sum_{j=1}^{n}(-1)^{j-1}
 {pn\brack pj}_q\mathcal E_{n-j}^{(p)}(q).
 \tag{6.2}\label{eq:generalized-recurrence}
\end{align}
The generalized $q$-Euler numbers considered here are closely related to
those introduced by Guo and Zeng~\cite{GuoZeng}.  Their convention is
\[
 \sum_{n\geq0}E_{pn}^{(p)}(q)\frac{x^{pn}}{(q;q)_{pn}}
 =
 \left(
 \sum_{n\geq0}\frac{x^{pn}}{(q;q)_{pn}}
 \right)^{-1},
\]
and comparison with \eqref{eq:generalized-genfun} gives
\begin{align}
 \mathcal E_n^{(p)}(q)=E_{pn\mid p}(q)
 =(-1)^nE_{pn}^{(p)}(q).
 \tag{6.2a}\label{eq:GuoZeng-sign}
\end{align}
Guo and Zeng proved a root-of-unity congruence for
$E_{kn}^{(k)}(q)$ with arbitrary block size $k$.  In particular, if
$\xi$ is a primitive $p$-th root of unity, their result with $k=p$,
$d=1$, and the lower index equal to zero gives
\[
 E_{pn}^{(p)}(\xi)=(-1)^n.
\]
It follows from \eqref{eq:GuoZeng-sign} that
\begin{align}
 E_{pn\mid p}(q)\equiv1\pmod{\Phi_p(q)}.
 \tag{6.2b}\label{eq:GuoZeng-first-order}
\end{align}
Thus the congruences established below may be viewed as higher
cyclotomic refinements of the $d=1$ case of the Guo--Zeng theorem:
instead of evaluating at a primitive $p$-th root of unity, we determine
the expansion through the third and fourth powers of its cyclotomic
polynomial.

Assume that $p$ is prime and put
\[
 \Phi=\Phi_p(q)=[p]_q,\qquad Q=q^{p^2}.
\]
The role of $1+q=\Phi_2(q)$ is now played by $\Phi_p(q)$.  The following
block form of the $q$-Ljunggren congruence is due, for $p\geq5$, to
Straub~\cite{Straub}; see also Zudilin~\cite{Zudilin}.  The cases $p=2,3$
follow by direct calculation:
\begin{align}
 {pn\brack pj}_q
 \equiv {n\brack j}_{Q}
 +\frac{j(n-j)}2\binom nj A_p(q)\Phi^2
 \pmod{\Phi^3},
 \tag{6.3}\label{eq:block-ljunggren}
\end{align}
where, modulo $\Phi$,
\begin{align}
 A_p(q)
 \equiv-\frac{p^2-1}{12}(q-1)^2
 \equiv
 \frac{{2p\brack p}_q-1-q^{p^2}}{\Phi^2}.
 \tag{6.4}\label{eq:Ap}
\end{align}

\begin{theorem}\label{thm:generalized-cubic}
For every prime $p$ and every $n\geq0$,
\begin{align}
 E_{pn\mid p}(q)
 \equiv q^{p^2\binom n2}
 +\binom n2A_p(q)[p]_q^2
 \pmod{[p]_q^3}.
 \tag{6.5}\label{eq:generalized-cubic}
\end{align}
\end{theorem}

\begin{proof}
We use induction on $n$.  The assertion is immediate for $n=0,1$, and its
case $n=2$ is \eqref{eq:Ap}.  Suppose it holds below $n$.  Substitute the
induction hypothesis and \eqref{eq:block-ljunggren} into
\eqref{eq:generalized-recurrence}.  The terms not containing $\Phi^2$ sum
to
\[
 \sum_{j=1}^{n}(-1)^{j-1}{n\brack j}_{Q}
 Q^{\binom{n-j}{2}}=Q^{\binom n2},
\]
by the finite $q$-binomial theorem.  Since $Q\equiv1\pmod\Phi$ and
${n\brack j}_{Q}\equiv\binom nj\pmod\Phi$, the coefficient multiplying
$A_p(q)\Phi^2$ is
\begin{align*}
 &\sum_{j=1}^{n}(-1)^{j-1}\binom nj
 \left\{\binom{n-j}{2}+\frac{j(n-j)}2\right\}\\
 &\qquad=\frac{n-1}{2}
 \sum_{j=1}^{n}(-1)^{j-1}\binom nj(n-j)
 =\binom n2.
\end{align*}
This proves \eqref{eq:generalized-cubic}.
\end{proof}

For $p=2$, one has $A_2(q)=-1$ modulo $1+q$, so
Theorem~\ref{thm:generalized-cubic} reduces to Liu's congruence
\eqref{eq:Liu}.  For example,
\[
 A_3(q)=2q\pmod{[3]_q},\qquad
 A_5(q)=-2(1-q)^2\pmod{[5]_q}.
\]

\subsection{The next cyclotomic correction}

We now continue the preceding argument modulo $\Phi^4$.  All coefficients
in this subsection are understood in the quotient ring
$\mathbb Q[q]/(\Phi)$.  Choose for $A_p(q)$ in \eqref{eq:Ap} its remainder
of degree less than $p-1$, and define
\begin{align}
 T_p(q)=\operatorname{rem}_{\Phi}
 \left(
 \frac{{2p\brack p}_q-1-Q-A_p(q)\Phi^2}{\Phi^3}
 \right)
 \tag{6.6}\label{eq:Tp}
\end{align}
and
\begin{align}
 U_p(q)=\frac p2(q-1)A_p(q)
 =-\frac{p(p^2-1)}{24}(q-1)^3.
 \tag{6.7}\label{eq:Up}
\end{align}
Thus $T_p$ is a central $q$-Wolstenholme-type quotient.

\begin{lemma}[Third-order block expansion]
\label{lem:third-order-block}
There exist $V_p(q),W_p(q)\in\mathbb Q[q]/(\Phi)$, independent of $n$ and
$j$, such that
\begin{align}
 {pn\brack pj}_q\equiv{}&{n\brack j}_{Q}
 +h_{n,j}A_p(q)\Phi^2\notag\\
 &+h_{n,j}\bigl(U_p(q)j(n-j)+V_p(q)n+W_p(q)\bigr)\Phi^3
 \pmod{\Phi^4},
 \tag{6.8}\label{eq:third-order-block}
\end{align}
where
\[
 h_{n,j}=\frac{j(n-j)}2\binom nj.
\]
Moreover,
\begin{align}
 U_p(q)+2V_p(q)+W_p(q)=T_p(q).
 \tag{6.9}\label{eq:UVW-T}
\end{align}
\end{lemma}

\begin{proof}
We specialize Zudilin's fourth-order congruence
\cite[Theorem~1.1]{Zudilin}.  For every integer $d>1$, it states that
\begingroup\small
\begin{align*}
 {dN\brack dj}_q\equiv{}&{N\brack j}_{q^{d^2}}\\
 &-j(N-j)\binom Nj(q^d-1)\Biggl\{
 N\mathcal H_{d-1}(q)+\frac{N(d-1)}2\\
 &\hspace{29mm}+\frac{(N+1)(d^2-1)}{24}(q^d-1)
 +\frac{\bigl(j(N-j)d-N-2\bigr)(d^2-1)}{48}
 (q^d-1)^2\Biggr\}\\
 &\hspace{29mm}\pmod{\Phi_d(q)^4},
\end{align*}
\endgroup
where
\[
 \mathcal H_{d-1}(q)=\sum_{r=1}^{d-1}\frac{q^r}{1-q^r}.
\]
This form is particularly well suited to the block coefficients in
\eqref{eq:generalized-recurrence}.

Take $d=p$ and write
\[
 \delta=q^p-1=(q-1)\Phi,\qquad
 \alpha=\frac{p^2-1}{24},\qquad
 \beta=\frac{p^2-1}{48}.
\]
Zudilin's expansion of the accompanying $q$-harmonic sum
\cite[Eq.~(12)]{Zudilin} implies
\[
 \mathcal H_{p-1}(q)+\frac{p-1}{2}
 =-\alpha\delta+O(\Phi^2).
\]
Consequently the following quotient is regular modulo $\Phi$:
\[
 R_p(q)=\operatorname{rem}_{\Phi}
 \left(
 \frac{\mathcal H_{p-1}(q)+(p-1)/2+\alpha\delta}{\Phi^2}
 \right).
\]
Substituting
\[
 \mathcal H_{p-1}(q)+\frac{p-1}{2}
 \equiv-\alpha\delta+R_p(q)\Phi^2
 \pmod{\Phi^3}
\]
in Zudilin's congruence, and using $\delta=(q-1)\Phi$, gives
\begin{align*}
 {pN\brack pj}_q\equiv{}&{N\brack j}_{Q}
 -j(N-j)\binom Nj\alpha(q-1)^2\Phi^2\\
 &-j(N-j)\binom Nj
 \left\{N(q-1)R_p
 +\beta\bigl(j(N-j)p-N-2\bigr)(q-1)^3\right\}\Phi^3
 \pmod{\Phi^4}.
\end{align*}
Because $A_p=-2\alpha(q-1)^2$, the coefficient of $\Phi^2$ is
$h_{N,j}A_p$.  After factoring out $h_{N,j}$, the coefficient of
$j(N-j)$ in the term of order $\Phi^3$ is
\[
 -2\beta p(q-1)^3
 =-\frac{p(p^2-1)}{24}(q-1)^3=U_p.
\]
The remaining terms are affine in $N$.  More explicitly, one may take
\[
 V_p=-2(q-1)R_p+2\beta(q-1)^3,
 \qquad
 W_p=4\beta(q-1)^3.
\]
This proves \eqref{eq:third-order-block}.  Finally, setting $N=2$ and
$j=1$ and comparing with the definition \eqref{eq:Tp} yields
\eqref{eq:UVW-T}.  Equivalently,
\[
 T_p=-4(q-1)R_p
 +\frac{(4-p)(p^2-1)}{24}(q-1)^3
 \quad\text{in }\mathbb Q[q]/(\Phi).
\]
\end{proof}

\begin{theorem}\label{thm:generalized-fourth}
For every prime $p$ and every $n\geq0$,
\begin{align}
 E_{pn\mid p}(q)\equiv{}&q^{p^2\binom n2}
 +\binom n2A_p(q)[p]_q^2\notag\\
 &+\binom n2\bigl(T_p(q)+U_p(q)(n-2)(n+1)\bigr)[p]_q^3
 \pmod{[p]_q^4}.
 \tag{6.13}\label{eq:generalized-fourth}
\end{align}
\end{theorem}

\begin{proof}
Define $B_n^{(p)}(q)$ by
\begin{align}
 E_{pn\mid p}(q)\equiv{}&Q^{\binom n2}
 +\binom n2A_p(q)\Phi^2+B_n^{(p)}(q)\Phi^3
 \pmod{\Phi^4}.
 \tag{6.14}\label{eq:Bnp-def}
\end{align}
Substitute \eqref{eq:third-order-block} and \eqref{eq:Bnp-def} into
\eqref{eq:generalized-recurrence}.  Comparing coefficients of $\Phi^3$
gives
\begin{align}
 B_n^{(p)}={}&\sum_{j=1}^{n}(-1)^{j-1}\binom njB_{n-j}^{(p)}\notag\\
 &+\sum_{j=1}^{n}(-1)^{j-1}h_{n,j}
 \left\{U_pj(n-j)+V_pn+W_p
 +2p(q-1)A_p\binom{n-j}{2}\right\}.
 \tag{6.15}\label{eq:Bnp-recurrence}
\end{align}
After division by $\binom nj$, the second summand is a polynomial in $j$
of degree at most four.  Hence the finite-difference argument of Section~4
shows that $B_n^{(p)}$ is a polynomial in $n$ of degree at most four.
Using \eqref{eq:UVW-T}, its first values are
\begin{align*}
 B_0^{(p)}&=B_1^{(p)}=0,\\
 B_2^{(p)}&=T_p,\\
 B_3^{(p)}&=3T_p+12U_p,\\
 B_4^{(p)}&=6T_p+60U_p.
\end{align*}
Therefore
\[
 \Delta^2B_0^{(p)}=T_p,\qquad
 \Delta^3B_0^{(p)}=\Delta^4B_0^{(p)}=12U_p.
\]
Newton interpolation yields
\begin{align*}
 B_n^{(p)}
 &=T_p\binom n2+12U_p\binom n3+12U_p\binom n4\\
 &=\binom n2\bigl(T_p+U_p(n-2)(n+1)\bigr),
\end{align*}
which proves \eqref{eq:generalized-fourth}.
\end{proof}

For $p=2$, one has
\[
 A_2=-1,\qquad U_2=2,\qquad T_2=1
 \quad\text{in }\mathbb Q[q]/(1+q).
\]
Thus Theorem~\ref{thm:generalized-fourth} gives
\[
 T_2+U_2(n-2)(n+1)=2n^2-2n-3,
\]
and recovers Theorem~\ref{thm:main}.  For $p=3$,
\[
 A_3=2q,\qquad U_3=-3-6q,\qquad T_3=-1-2q,
\]
so
\begin{align}
 B_n^{(3)}(q)=-(1+2q)\binom n2(3n^2-3n-5).
 \tag{6.16}\label{eq:p3-B}
\end{align}
For $p=5$,
\[
 A_5=-2(1-q)^2,\qquad U_5=5(1-q)^3,
\]
and
\[
 T_5=1-11q+7q^2-5q^3.
\]
Consequently,
\begin{align}
 B_n^{(5)}(q)=\binom n2\bigl\{&5n^2-5n-9
 +(-15n^2+15n+19)q\notag\\
 &+(15n^2-15n-23)q^2
 +(-5n^2+5n+5)q^3\bigr\}.
 \tag{6.17}\label{eq:p5-B}
\end{align}

\section{Concluding remarks}

We have developed an analytic method for higher $(1+q)$-adic congruences of
the $q$-secant inversion enumerator.  Its principal ingredients are the
reciprocal generating function, local expansions of Gaussian coefficients,
finite differences, and Newton interpolation.  Besides the fourth-order
congruence in Theorem~\ref{thm:main}, the recurrence in Section~5 gives an
effective procedure for computing the expansion modulo $(1+q)^K$ for any
prescribed $K$.

Liu's proof of \eqref{eq:Liu} is combinatorial: natural pairs generate
switching orbits, and three free pairs account for divisibility by
$(1+q)^3$.  It would be interesting to find combinatorial interpretations
of the two higher correction coefficients
\[
 \binom n2(2n^2-2n-3)
\]
and
\[
 -\frac{n(n-1)(n-2)(24n^3-65n-65)}{24}.
\]
Such interpretations should presumably arise from a finer analysis of
switching orbits with three or four free natural pairs.

For the generalized $q$-Euler numbers of Sagan and Zhang, the natural local
parameter is $\Phi_p(q)=[p]_q$.  Zudilin's fourth-order expansion supplies
the required local information for each Gaussian block, while the
Sagan--Zhang recurrence and Newton interpolation determine its cumulative
effect on $E_{pn\mid p}(q)$.  Theorems~\ref{thm:generalized-cubic} and
\ref{thm:generalized-fourth} show that the resulting fourth-order
correction is determined by the central quotient $T_p(q)$ in
\eqref{eq:Tp}, for which the proof of
Lemma~\ref{lem:third-order-block} also gives a $q$-harmonic representation. 
The first-order root-of-unity congruences of Guo and
Zeng~\cite{GuoZeng} hold for generalized $q$-Euler numbers with an
arbitrary block size $k$.  This provides additional motivation for
extending the higher-order results of Section~6 to composite block
sizes.  In that setting one must work separately at the cyclotomic
factors of
\[
 [k]_q=\prod_{\substack{d\mid k\\d>1}}\Phi_d(q),
\]
since a congruence modulo a single $\Phi_d(q)^4$ does not immediately
produce a congruence modulo $[k]_q^4$.

Zudilin's block congruence holds modulo $\Phi_d(q)^4$ for every $d>1$.
This suggests extending the generalized Euler result to composite block
sizes modulo $\Phi_d(q)^4$.  Notice, however, that this does not immediately
give a congruence modulo $[d]_q^4$, because $[d]_q$ has several cyclotomic
factors when $d$ is composite.

The complementary theory of generalized $q$-tangent subsequences
$E_{pn+i\mid p}(q)$, including a $p$-cyclotomic extension of Foata's
divisibility theorem and generalized Catalan formulas for normalized
quotients, will be developed in a forthcoming paper.

\end{document}